\documentclass[12pt]{amsart}
\usepackage{amsmath}
\usepackage{amssymb,latexsym}
\usepackage{amsthm,amscd}
\theoremstyle{plain}
\newtheorem{theorem}{Theorem}

\newtheorem{corollary}[theorem]{Corollary}
\newtheorem{proposition}[theorem]{Proposition}
\theoremstyle{definition}

\newtheorem{remark}{Remark}
\newtheorem{example}{Example}
\begin{document}
\title{Flocks of Cones: Star Flocks}
\author{William Cherowitzo}
\address{Department of Mathematical and Statistical Sciences, University of Colorado Denver, Campus Box 170, P.O. Box 173364, Denver, CO  80217-3364, U.S.A.}
\email{william.cherowitzo@ucdenver.edu}

\subjclass{Primary 51E20, 51E21;  Secondary 05B25}


\keywords{Flocks, Cones, Star Flocks, R\'edei Blocking Sets, Bilinear Flocks}
\begin{abstract}
The concept of a flock of a quadratic cone is generalized to
arbitrary cones. Flocks whose planes contain a common point are
called star flocks. Star flocks can be described in terms of
their coordinate functions. If the cone is ``big enough'', the
star flocks it admits can be classified by means of a connection
with minimal blocking sets of R\'edei type. This connection can also be used to obtain examples of bilinear flocks of non-quadratic cones.
\end{abstract}

\maketitle
\section{Introduction}
This is the second (see \cite{WEC2}) in a series of articles
devoted to providing a foundation for a theory of flocks of
arbitrary cones in $PG(3,q)$. The desire to have such a theory
stems from a need to better understand the very significant and
applicable special case of flocks of quadratic cones in
$PG(3,q)$. Flocks of quadratic cones have connections with
several other geometrical objects, including certain types of
generalized quadrangles, spreads, translation planes, hyperovals
(in even characteristic), ovoids, inversive planes and
quasi-fibrations of hyperbolic quadrics. This rich collection of
interconnections is the basis for the strong interest in such
flocks.

\section{Cones and Flocks}

Let $\pi_0$ be a plane and $V$ a point not on $\pi_0$ in
$PG(3,q)$. Let $\mathcal{S}$ be any set of points in $\pi_0$
(including the empty set). A \emph{cone}, $\Sigma =
\Sigma(V,\mathcal{S})$ is the union of all points of $PG(3,q)$ on
the lines $VP$ where $P$ is a point of $\mathcal{S}$. $V$ is
called the \emph{vertex} and $\mathcal{S}$ is called the
\emph{carrier} of $\Sigma$. $\pi_0$ is the \emph{carrier plane}
and the lines $VP$ are the \emph{generators} of $\Sigma$. In the
event that $\mathcal{S} = \emptyset$ we call $\Sigma$ the
\emph{empty cone} and by convention consider it to consist of
only the point $V$.

    A \emph{flock of planes} in $PG(3,q)$ is any set of $q$
\emph{distinct} planes of $PG(3,q)$. As $q$ planes can not cover
all the points of $PG(3,q)$, there always are points of the space
which do not lie in any of the planes in a flock of planes. If
$\Sigma$ is a cone of $PG(3,q)$, then a flock of planes,
$\mathcal{F}$, is said to be a \emph{flock of }$\Sigma$ when the
vertex of $\Sigma$ lies in no plane of $\mathcal{F}$ and no two
planes of $\mathcal{F}$ intersect at a point of $\Sigma$. Any
flock of planes is a flock of a cone, possibly only the empty
cone. In general, however, a given flock of planes will be a flock
of several cones.  In the literature on flocks of quadratic
cones, the approach is always to consider a fixed quadratic cone
and study the flocks of that cone. We will change the viewpoint
and consider, for a fixed flock of planes, the various cones of
which it is a flock. In the sequel we shall refer to a flock of
planes simply as a \emph{flock} and it shall be understood that
it is always a flock of a cone, even if the cone is not
explicitly indicated. Furthermore, we shall always assume, unless explicitly stated otherwise, that a flock is a flock of a non-empty cone.

    Let $\mathcal{F}$ be a flock. We can introduce coordinates in
$PG(3,q)$ so that the plane $x_3 = 0$ is one of the planes of the
flock and the point $V = ( 0,0,0,1 )$ is not in any
plane of the flock. Since $V$ is not
in any plane of $\mathcal{F}$, each of the planes of this flock
has an equation of the form $Ax_0 + Bx_1 + Cx_2 - x_3 = 0$. We parameterize the planes of $\mathcal{F}$
with the elements of $GF(q)$ in an arbitrary way except that we
will require that $0$ is the parameter assigned to the plane $x_3
= 0$. We can now describe the flock as, $\mathcal{F} = \{\pi_t\colon f(t)x_0 + g(t)x_1 + h(t)x_2 - x_3 = 0
\mid t \in GF(q)\}$ with $\pi_0 \colon x_3 = 0$.
The functions $f,g \text{ and } h$ are called the \emph{coordinate
functions} of the flock. Note that the requirement on the
parameter $0$ means that $f(0) = g(0) = h(0) = 0$. If $f,g \text{
and } h$ are the coordinate functions of the flock $\mathcal{F}$
we shall write $\mathcal{F} = \mathcal{F}(f,g,h)$. We remark that
the coordinate functions of a flock depend on the
parameterization of the flock.

    As all cones under consideration have vertex $V = (
0,0,0,1 )$ and the plane $\pi_0$ as the
carrier plane, a cone is determined when its
carrier $\mathcal{S}$, a point set in $\pi_0$, is specified.
Given a flock $\mathcal{F}$, there is a largest set
$\mathcal{S}_0$ of $\pi_0$ such that $\mathcal{F}$ is a flock of
the cone with carrier $\mathcal{S}_0$. This cone is called the
\emph{critical cone} of $\mathcal{F}$. If $\mathcal{C}$ is any
subset of the carrier of the critical cone of a flock
$\mathcal{F}$, then clearly $\mathcal{F}$ is also a flock of the
cone with carrier $\mathcal{C}$. Thus, determining the critical
cone of a flock implicitly determines all cones for which this
flock of planes is a flock.

\begin{theorem} \label{Th:herd}
A point $( a,b,c,0 )$ is in the carrier of the critical cone of the flock $\mathcal{F} = \mathcal{F}(f,g,h)$ in $PG(3,q)$ if and only if the function $t \mapsto af(t) + bg(t) + ch(t)$ is a permutation of $GF(q)$.
\end{theorem}

\begin{proof}
A point on the line $\langle (0,0,0,1), (a,b,c,0) \rangle$ other than the vertex has coordinates $( a,b,c,\lambda )$ with $\lambda \in GF(q)$. Such a point is on the plane $\pi_t$ of $\mathcal{F}$ if and only if $\lambda = af(t) + bg(t) + ch(t)$. No two planes of $\mathcal{F}$ will meet at the same point of this line if and only if $t \mapsto af(t) + bg(t) + ch(t)$ is a permutation of $GF(q)$. Thus, under this condition, the point $( a,b,c,0 )$ will be in the carrier of the critical cone of $\mathcal{F}$.
\end{proof}

    The critical cone of a flock may be fairly ``small''. Besides
the empty cone, we will consider cones whose carriers consist of
collinear points as being ``small''. Cones of this type are
called \emph{flat cones}. For the most part, we shall regard
flocks whose critical cones are flat as being uninteresting. For any nonempty set $S$ and any point $P$ of a projective plane, define $w_S(P)$ to be the number of lines through $P$ which contain an element of $S$. In the projective plane $\Pi$ we can define the \textit{width} of a set $S$ to be $W_S = min \{ w_S(P) \mid P \in \Pi \}$. Clearly, $W_S = 1$ if and only if $S$ consists of a set of collinear points. If $S$ is an oval in a projective plane of order $q$ ($q+1$ points, no three of which are collinear), then $W_S = \frac{q+1}{2}$ if $q$ is odd or $W_S = \frac{q+2}{2}$ if $q$ is even. This can be written as $W_S = \lfloor \frac{q+2}{2} \rfloor$, independent of the parity of $q$. Using this idea we can provide an admittedly crude classification of critical cones.
In $PG(3,q)$, if $S$ is the set of points of a carrier of a cone, then if $W_S < \lfloor \frac{q+2}{2} \rfloor$, we call the cone a \emph{thin cone}. Cones which are not
thin are called \emph{wide} and a wide cone with at least $q+1$
points in a carrier are called \emph{thick cones}. The class of
thick cones contains the quadratic cones as well as all cones
whose carrier contains any oval.

\section{Definitions and Basic Properties}
 A \emph{linear flock} is a flock whose planes share a
common line. Any cone whose carrier is not a proper blocking set
(a set of points which contains no line and which every line
intersects) in its carrier plane admits linear flocks. Linear
flocks can easily be characterized by their coordinate functions.

\begin{theorem} \label{Th:6.1}
A flock is a linear flock if and only if the three coordinate
functions are scalar multiples of each other.
\qed
\end{theorem}

 A \emph{star flock} is a flock whose
planes share a common point and a \emph{proper} star flock is one
for which this common point is unique. The common point of a
proper star flock is called the \emph{star point}. Linear flocks
are clearly star flocks, but not proper star flocks.

 We may also characterize star flocks in terms of their coordinate functions.

\begin{theorem} \label{Th:7.1}
 A flock is a star flock if and only if  the coordinate functions are linearly
dependent over GF(q).
\qed
\end{theorem}

Since a flock is a set of planes of $PG(3,q)$ it is natural to define the equivalence of flocks
without reference to the cones that they are flocks of. While there are several ways to do this, the following has proved to be most useful. Two flocks, $\mathcal{F}_1$ and $\mathcal{F}_2$ are \emph{equivalent} if they are in the same orbit of the group $P{\Gamma}L(4,q)_{V,x_3=0}$. That is, there is a collineation of $PG(3,q)$ fixing the point $V = ( 0,0,0,1 )$ and stabilizing the plane $x_3 = 0$ which maps the planes of $\mathcal{F}_1$ to those of $\mathcal{F}_2$.

\begin{theorem} \label{Th:11.1.1}
 A star flock is equivalent to a flock $\mathcal{F}(t,g(t),0)$, where $0$ denotes the function which is identically zero.
\end{theorem}

\begin{proof}
Let $P$ be a point in the carrier of the critical cone and $Q$ the star point of a star flock. By using a collineation of $PG(3,q)$ which fixes $(0,0,0,1)$ and stabilizes $x_3 = 0$, we may assume that $P = (1,0,0,0)$ and $Q = (0,0,1,0)$. If the star flock is given by $\mathcal{F} = \mathcal{F}(f,g,h)$, then by Theorem \ref{Th:herd}, $t \mapsto f(t)$ is a permutation. We can re-parameterize the planes of the flock so that $f(t) = t$. Since all the planes of the flock pass through $Q$, we must have $h \equiv 0$.
\end{proof}

\section{Star Flocks of Wide Cones}

 Viewing flocks in a dual setting has been
an effective technique in the classical quadratic cone situation.
It is especially useful in studying star flocks.

 Let $\mathcal{F}$ be a flock with critical cone $\Sigma = \Sigma(V,\mathcal{S})$
in $PG(3,q)$. By passing to the dual, $\mathcal{F}$ becomes a set
of $q$ points in $PG(3,q)$ and the cone is the set of all planes
passing through a set of lines in the plane corresponding to $V$,
with the property that no line determined by a pair of the $q$
points lies in any of these planes.

 If $\mathcal{F}$ is a proper star flock, then the corresponding $q$
points in the dual are coplanar. We fix some notation for this situation.
Let $\mathcal{F} = \mathcal{F}(t,g(t),0)$ be a star flock of the cone
$\Sigma = \Sigma(V,\mathcal{S})$ with $V = ( 0,0,0,1)$
and the planes of $\mathcal{F}$ given by $tx_0 + g(t)x_1 +
 - x_3 = 0$. For the sake of clarity, we employ the
notational device of using $(X,Y,Z,W)$ to refer to the generic
coordinates of a point when viewed in the dual setting. Using the
duality $(x,y,z,w) \leftrightarrow [x,y,z,-w]$, the flock becomes a set of points $D_F =
\{(t, g(t), 0, 1) \colon t \in GF(q)\}$. $V$ becomes the
plane with equation $W = 0$. The points on a generator of
$\Sigma(V,\mathcal{S})$ become the set of all planes passing
through a line in $W = 0$. The set of lines in $W = 0$,
corresponding to the generators of $\Sigma$ will
be denoted by $D_G$, and the set of all planes passing through
the lines of $D_G$ will be denoted by $D_{\Sigma}$. The condition that lines formed by pairs of points of
$D_F$ do not lie in the planes of $D_{\Sigma}$ is equivalent to
the condition that these lines do not intersect the lines of
$D_G$. The $q$ points of $D_F$ lie in the plane $Z = 0$, and since $\mathcal{F}$ is proper, they are not all collinear.
Let $m$ be the line of intersection of the planes $W = 0$ and $Z = 0$.

  The $q$ points of $D_F$ lie
in the affine plane obtained by removing the line $m$ from
$Z = 0$. Since $\mathcal{F}$ is a flock, the line joining any two
of these points can not intersect $W=0$ at a point on any line of
$D_G$. Thus, the set of points on $m$ at which these lines meet
$m$ is disjoint from the set of points on $m$ at which the lines
of $D_G$ meet $m$. Let $M$ denote the set of points of $m$ at
which lines determined by pairs of points of $D_F$ meet $m$, and
let $N = |M|$. $N$ can be thought of as the number of ``slopes''
determined by the set of $q$ points of $D_F$.

A \emph{blocking set} in a projective plane is a set of points in the plane which every line intersects. Let $n$ be the largest number of collinear points in a blocking set. In a plane of order $q$, a blocking set of size $q+n$ is called a \emph{R\'{e}dei blocking set}.

\begin{proposition} \label{Pr:Redei}
In $Z = 0$ the points of $D_F \cup M$ form a R\'{e}dei blocking set.
\end{proposition}

\begin{proof}
Consider a point of $m$ which does not lie in $M$. The $q$ lines through this point other than $m$ must each contain exactly one point of $D_F$. Therefore, $D_F \cup M$ is a blocking set. Since the points of $D_F$ are not collinear, no collinear set of points of $D_F$ can contain more than $N-1$ points. Thus $D_F \cup M$ is of R\'edei type of size $q+N$.
\end{proof}



 The problem of determining the number of slopes
determined by $q$ points in an affine plane has been well studied
(\cite{LR:70},\cite{LoSc:81},\cite{BlBrSz:95},\cite{BlBaBrStSz:99},\cite{SB:02})
due to the connection with blocking sets of R\'edei type.
This problem is completely settled in the cases we are interested
in and the relevant theorem (paraphrased in our terminology) is
due to Ball \cite{SB:02} who settled two open cases of the main
theorem first given by Blokhuis, Ball, Brouwer, Storme and
Sz\"{o}nyi \cite{BlBaBrStSz:99}.

%
%
%
%

\begin{theorem}[\cite{SB:02}] \label{Th:11.1.2}
Let $D_F$ consist of $q$ points of $PG(2,q)$ whose coordinates are
given by $(t,g(t),1)$, $t \in GF(q)$, where $g$ is a permutation
polynomial with $g(0) = 0$. If $q = p^n$ for some prime $p$, let
$e$ be the largest integer so that any line of the plane
containing at least two points of $D_F$ contains a multiple of
$p^e$ points of $D_F$. If $N$ is the number of ``slopes''
determined by the points of $D_F$ then we have one of the
following:
\renewcommand{\theenumi}{\roman{enumi}}
\begin{enumerate}
\item  $e = 0$ and $\frac{q+3}{2} \le N \le q+1$,

\item  $e \mid n$ and $p^{n-e} + 1 \le N \le
\frac{q-1}{p^e -1}$,

\item $e = n$ and $N = 1$.
\end{enumerate}
Moreover, if $p^e > 2$, then g is a $p^e$-linearized polynomial.
\qed
\end{theorem}

 A \emph{$p^e$-linearized polynomial} is one of the form:
\begin{equation*}
g(t) = \sum_{i=0}^{k-1} \alpha_i t^{p^{ie}}, \text{ where }
\alpha_i \in GF(p^n), n = ke.
\end{equation*}

\noindent In this theorem, all the bounds for $N$ are sharp.
The $p^e$-linearized permutation polynomials form a group under
composition mod($x^q-x$). This is the Betti-Mathieu group (see \cite{LiNi:97}) and it
is known to be isomorphic to $GL(k,p^e)$ where $q = p^{ek}$. Thus,

\begin{proposition} \label{Pr:13}
Let $q = p^n$ and $s = p^e$ with $n = ek$. Then the number of
monic $s$-linearized permutation polynomials over $GF(q)$ is:
\begin{equation} \label{Eq:13}
s^{\frac{k(k-1)}{2}}\prod_{i=1}^{k-1}  (s^i-1).
\end{equation}
\qed
\end{proposition}

 Theorem \ref{Th:11.1.2}  has the following implication for star flocks of wide cones:

\begin{theorem} \label{Th:11.1.3}
If $q = p^n$ with $p$ a prime, then $\mathcal{F}$ is a proper
star flock with a wide critical cone if and only if there exists a
coordinatization such that $\mathcal{F} =
\mathcal{F}(t,g(t),0)$ where $g$ is a non-linear
$p^e$-linearized permutation polynomial for some $e \mid n$ with
$e < n$. Furthermore, the critical cone of this star flock contains the points $(x, f(x), 1, 0)$
for $x \ne 0$ if and only if
\begin{equation} \label{Eq:12}
    f(x)g(t) + xt = 0 \mbox{ has no solutions with } xt \ne 0 .
\end{equation}

\end{theorem}

\begin{proof}
 By Theorem \ref{Th:11.1.1} we may assume that $\mathcal{F} = \mathcal{F}(t,g(t),0)$ is a proper star flock of a wide cone with star point
$P = (0,0,1,0)$ and $g$ a non-linear permutation polynomial
(non-linearity follows from Theorem \ref{Th:6.1}).

 Since the cone is wide, there are at least $\lfloor \frac{q+2}{2}
\rfloor$ lines in $x_3 = 0$ through $P$ which contain the points of
the critical cone of this flock. Thus, there are at least $\lfloor
\frac{q+2}{2} \rfloor$ planes determined by these lines and the
vertex of the cone. We now consider the dual setting and use the
notation introduced at the beginning of this section. These
planes correspond to points in $W=0$ which are also in $Z=0$,
i.e., points on $m$. Since these planes contain points of the
carrier, the points on $m$ that they correspond with are points
where lines of $D_G$ meet $m$. Thus, for a wide cone, we must
have $N \le q+1 - \lfloor \frac{q+2}{2} \rfloor = \lfloor
\frac{q+1}{2} \rfloor$. Since the star flock is proper, we must
also have $N> 1$. We can identify $Z=0$ with the projective
plane $\Pi$ obtained by suppressing the third coordinates.
Thus, the points of $D_F$ have coordinates $(t,g(t),1)$ in
$\Pi$. By Theorem \ref{Th:11.1.2}, if $p^e = 1$ or $2$, the
number of slopes determined by $D_F$ does not satisfy $1 < N \le
\lfloor \frac{q+1}{2} \rfloor$. Thus, $p^e > 2$, and we have that
$g$ is a $p^e$-linearized polynomial.

 Now, suppose that g is a non-linear $p^e$-linearized polynomial for some $e \mid n$ with $e < n$. Such
functions are additive (i.e., $g(t+s) = g(t) + g(s)$), so to
calculate $N$ we need only calculate the number of distinct values
of $g(t)/t$ for $t \ne 0$. Let $K$ be the proper subfield of
$GF(q)$ of order $p^e$. For each $c \in K$ we have $g(ct) =
cg(t)$. Thus, for $c \in K \setminus \{0\}, g(ct)/ct = g(t)/t$,
and the number of distinct non-zero values of $g(t)/t$ is at most
$(q-1)/(p^e-1)$. Therefore, $N \le (q-1)/(p^e-1) + 1 < \lfloor
\frac{q+1}{2} \rfloor$ if $p^e > 2$. We obtain a star flock of a wide cone
from such a $g$, and it is proper since $g$ is not linear.

Finally, For each $x
\ne 0$, $f(x)g(t) + xt$ is a $p^e$-linearized polynomial (in $t$). Such
polynomials represent permutations if and only if they vanish
only at $t=0$ (Dickson \cite{LED:01}). Thus, (\ref{Eq:12})
insures that each of the points $(x,f(x),1,0)$ is in the critical
cone of $\mathcal{F}$ by Theorem \ref{Th:herd}.
\end{proof}

%

\begin{example} \label{Ex:1}
Let $q = p^h, \quad p$ an odd prime, and $h > 1$. Let $\sigma =
p^i, 1 \le i \le h-1$. Then $t^{\sigma}$ is a
$p^{gcd(i,h)}$-linearized polynomial. In $x_3 = 0$ the points $(x,
-m/x, 1, 0), (1,0,0,0)$ and $(0,1,0,0)$ form a conic ($xy=-m$).
$\frac{-m}{x}t^{\sigma}  + xt = xt(\frac{-mt^{\sigma -1}}{x^2} +
1)$ will have no solutions with $xt\ne 0$ if and only if $m$ is a
non-square in $GF(q)$.  Thus, $\mathcal{F}(t, t^{\sigma}, 0)$
will be a (proper) star flock of the quadratic cone $xy = -m$
when $m$ is a non-square. These star flocks are known as the
\emph{Kantor-Knuth} (or K1) flocks of a quadratic cone. These examples come from
a special case ($k=0$) of the fact that the points $(x,f(x),1,0), x \ne 0$ where
$f(x) = \frac{-m}{x^{2k+1}}$, $0 \le k \le
\frac{q-1}{2}$ are in the critical cone of this star flock when $m$ is a non-square.
\end{example}

\begin{corollary} \label{Co:11.1.1}
If $q$ is a prime, then all star flocks of wide cones in
$PG(3,q)$ are linear.
\end{corollary}

\begin{proof}
 Since a prime field has no proper subfields, there are
no $p^e$-linearized polynomials other than the linear ones, so by
Theorem \ref{Th:11.1.3} there are no proper star flocks.
\end{proof}

\begin{corollary} \label{Co:11.1.2}
If $q = 2^p$, with $p$ a prime, then all star flocks of wide
cones in $PG(3,q)$ are linear.
\end{corollary}

\begin{proof}
If a star flock were not linear in this situation, then we would
have $p^e = 2$ which is ruled out since it gives too large an $N$
for a wide cone.
\end{proof}

 We examine two classes of examples of proper star flocks of wide cones which
correspond to extreme values of $N$. Let $E = GF(q) = GF( p^n)$
with a subfield $K = GF(p^e)$, then the permutation function $g(t)
= k(t^{p^e} - ct)$, where $c \ne {\beta}^{p^e - 1}$ with $k,
\beta \in GF(q)^*$, gives $N = (q-1)/(p^e -1)$ and for $p^e
> 2$ we have $N < \lfloor \frac{q+1}{2} \rfloor $, and so, $\mathcal{F}(t,g(t),-(at + bg(t)))$ for
$a,b \in GF(q)$ is a star flock of a wide cone. These examples
include those of Example \ref{Ex:1} in odd characteristic, so we
shall call them (or any flocks equivalent to them, see
\cite{WEC2}) \emph{Kantor-Knuth star flocks}. Under the same
assumptions about $q$ and $e$, the function $g(t) =
k_1(Tr_{E/K}(\frac{t}{k_2}) - ct)$, where $c \ne
\frac{1}{\beta}Tr_{E/K}(\frac{\beta}{k_2})$,  $k_1, k_2, \beta \in
GF(q)^*$ and $Tr_{E/K}$ is the relative trace function from $E$
onto $K$, gives $N = p^{n-e} + 1$, and again, if $p^e > 2, N <
\lfloor \frac{q+1}{2} \rfloor$. We call these proper star flocks
of wide cones, \emph{Holder-Megyesi star flocks} (L. Megyesi gave
the example in the R\'edei blocking set context and L. Holder,
independently, gave the example in the conic blocking set
context (see \cite{LH:04}).

\begin{corollary} \label{Co:11.1.3}
If $q=p^2$ for any prime $p$, or $p = 4$, then all star flocks of
wide cones in $PG(3,q)$ are either linear or Kantor-Knuth star
flocks.
\end{corollary}

\begin{proof}
All flocks of wide cones in $PG(3,4)$ are linear (see
\cite{WEC2}) so we may assume that $p > 2$. Up to scalar
multiples the only $p$-linearized permutation polynomials are of
the form $t^p - ct$, where $c \ne {\beta}^{p-1}$ for any $\beta
\in GF(q)^*$ or $t$. In the case of $PG(3,16)$, the 2-linearized
polynomials do not give star flocks of wide cones, so the only
$s$-linearized polynomials giving star flocks of wide cones have
$s = 4$.
\end{proof}

\begin{example} \label{Ex:2}
In PG(2,16) there are two projectively inequivalent hyperovals, the
hyperconic (a.k.a. regular hyperoval) and the Lunelli-Sce hyperoval. The cone over a hyperconic (quadratic
cone) admits only linear star flocks (Thas \cite{JAT:87}), but a cone over the
Lunelli-Sce hyperoval admits both linear and
Kantor-Knuth star flocks.
   Let $\lambda$ be a primitive element of $GF(16)$ satisfying
$\lambda^4 = \lambda + 1$. The point sets of $x_3 = 0$ given by
$\mathcal{H}_i = \{(x,f_i(x),1,0)\colon x \in GF(16)\} \cup
\{(0,1,0,0), (1,0,0,0)\}$ with $i=1,2$, where
\begin{gather*}
f_1(x) = \lambda^{13}x^{14} + \lambda^3 x^{12} + \lambda^6 x^{10}
+ x^8 + \lambda^6 x^6 + \lambda^3 x^4 + \lambda^{13}x^2 \makebox{
and } \\
f_2(x) = \lambda^4 x^{14} + \lambda^{10} x^{12} + \lambda^{11}
x^{10} + \lambda^{11} x^8 + x^6 + \lambda^2 x^4,
\end{gather*}
represent Lunelli-Sce hyperovals which intersect in nine points,
the maximum number of points in the intersection of two
hyperovals in $PG(2,16)$. Among the points of intersection are
the points $(0,0,1,0),$ $(1,1,1,0),$ $(\lambda^{10}, \lambda^{13},
1,0),$ $(\lambda^8, \lambda^2,1,0),$ $(\lambda^{14}, \lambda^5,
1,0)$ and $(\lambda^{13}, \lambda^{10},1,0)$. The slopes of the
lines through $(0,0,0,1)$ and each of the other five points form
a set of values $\{\lambda^{3i} \colon 0 \le i \le 4 \}$, i.e.,
the non-zero cubes in $GF(16)$. The symmetric difference of these
two hyperovals, $\mathcal{H}_1 \triangledown \mathcal{H}_2$ is
another Lunelli-Sce hyperoval (see \cite{BrCh:00}) with these five
lines as exterior lines through the point $(0,0,1,0)$. The proper
Kantor-Knuth star flock, $\mathcal{F}(t,t^4,0)$, with star point
$(0,0,1,0)$ contains $\mathcal{H}_1 \triangledown \mathcal{H}_2$
in the carrier of its critical cone.
\end{example}

\begin{corollary} \label{Co:11.1.4}
If $q = p^3$ for any prime $p$ or $p = 4$, then all star flocks of
wide cones in $PG(3,q)$ are either linear, Kantor-Knuth or
Holder-Megyesi star flocks.
\end{corollary}

\begin{proof}
If $p = 2$, then a star flock of a wide cone must be linear by
Corollary \ref{Co:11.1.2}, so we may assume that $p > 2$. By
Sherman \cite{BFS:02}, only two values of $N$ occur when $q =
p^3$, namely $p^2 + 1$ and $p^2 + p + 1$. Furthermore, when $N =
p^2 + 1$ then $g(t) = k_1(Tr_{E/K}(\frac{t}{k_2}) - ct)$, where
$c \ne \frac{1}{\beta}Tr_{E/K}(\frac{\beta}{k_2})$,  $k_1, k_2,
\beta \in GF(q)^*$. There are $(p^2 + p + 1)(p^3 - p^2 - 1)$
projectively equivalent monic polynomials of this type
(Holder-Megyesi). There are $(p^3 - p^2 - p -1)^2$ projectively
equivalent monic polynomials of the form $t^{p^2} - at^p - ct$
with $a \ne {\beta}^{p^2-p}$ for any $\beta \in GF(q)^*$ which
are permutation polynomials. These are projectively equivalent to
the $p^3 - p^2 - p - 1$ monic permutations of the form $t^p - at$
with $a \ne {\beta}^{p-1}$ for any $\beta \in GF(q)^*$
(Kantor-Knuth). Together with the single monic linear function,
we have accounted for all the monic $p$-linearized permutation
polynomials (Proposition \ref{Pr:13}).
\end{proof}

\begin{remark}
In principle we could continue in this vein and classify all star
flocks of wide cones. Indeed, it would be easy enough to do the
next case of $q = p^4$, since Sherman \cite{BFS:02} has already
determined the spectrum of values of $N$ in this case.
\end{remark}

\section{Bilinear Star Flocks}

A \emph{bilinear flock} is a flock in which each plane passes through at least one of two distinct lines of $PG(3,q)$ (carrier lines). If the carrier lines of a bilinear flock meet, the flock is a proper star flock. Bilinear flocks in $PG(3,K)$ with $K$ infinite have been studied (Biliotti and Johnson \cite{BiJo:99}), but no bilinear flocks of quadratic cones are known in the finite case.

In the dual setting for $PG(3,q)$, a proper bilinear star flock corresponds to a R\'edei blocking set whose $q$ affine points lie on two distinct lines. In this situation, using the notation of the previous section, there will exist lines containing exactly two points of $D_F$ and so, $p^e \le 2$. It follows from the proof of Theorem \ref{Th:11.1.3} that any cone of such a flock must be thin.

\begin{theorem} \label{Th:nobi}
In $PG(3,q)$ there are no proper bilinear star flocks of wide cones.
\qed
\end{theorem}

This result is sharp as the following examples show.

\begin{example}
Let $q$ be odd. A well known R\'edei blocking set in $PG(2,q)$ is the \emph{projective triangle} of side $(q+3)/2$ (see \cite{JWPH:79}). This consists of $3(q+1)/2$ points with $(q+3)/2$ points on each side of a triangle (including the vertices) such that the line joining any two points of the set on different sides of the triangle meets the third side at a point of the set. A representation of a projective triangle which includes the point $(0,0,1)$ has affine points of the form $(x,x^{\frac{q+1}{2}},1)$ and infinite points $(1,\frac{1+z}{1-z},0) \cup (1,-1,0) \cup (1,1,0)$ for each non-square $z \in GF(q)$. Note that the affine points of the set lie on the two lines $y = x$ (for square $x$) and $y = -x$ (for non-square $x$). This representation gives rise to the proper star flock $\mathcal{F}(t,t^{\frac{q+1}{2}},0)$ which can only be a flock of a thin cone. In particular, since the star point is $(0,0,1,0)$ the critical cone consists only of points other than $(0,0,1,0)$ on the lines $y = cx$ where $\frac{1-c}{1+c}$ is a non-zero square. There are $\frac{q-1}{2}$ such values of $c$ if $-1$ is not a square in $GF(q)$ and $\frac{q-3}{2}$ otherwise. The point $(0,1,0,0)$ is in the critical cone, and hence the points of the line $x = 0$ other than $(0,0,1,0)$ if and only if $-1$ is a square in $GF(q)$, that is, $q \equiv 1 \pmod{4}$. Therefore, for any odd $q$, the width of the base of this cone is $\frac{q-1}{2} < \lfloor \frac{q+2}{2} \rfloor$ and so the cone is thin, but any larger cone would be wide.

A special case of this example occurs in $PG(2,q^2)$ for $q$ an odd prime power. The map $x \mapsto x^q$ is an involutory automorphism of $GF(q^2)$. The flock $\mathcal{F}(t,At^{\frac{q^2 + 1}{2}},0)$, where $A^q = -A$ has a critical cone which contains the points of $x^qy = z^{q+1}$ in the carrier plane $w = 0$. In the dual setting, the points of $Z = 0$ on the lines of $D_G$ are (with Z coordinate suppressed) $(0,1,0)$ and $\{(1,\alpha,0) \mid \alpha \in GF(q)\}$. The points of $W=0 $ on the lines determined by pairs of points of $D_F$ are (again with Z coordinate suppressed): $(1,A,0)$,$(1,-A,0)$ and $\{(1,A(\frac{1+z}{1-z}),0) \mid z \text{ is a nonsquare in }GF(q^2) \}$. All of these points lie on the line $m = \{W=0\} \cap \{Z=0\}$ and we wish to determine  $A$ so that the sets are disjoint. Clearly, $A$ can not lie in $GF(q)$. The condition that $A(\frac{1+z}{1-z}) \in GF(q)$ is (with $\lambda$ a primitive element of $GF(q^2)$), $(A^{q-1}-1)(\lambda^{(q+1)(2k+1)} - 1) = (A^{q-1}+1)(\lambda^{(q-1)(2k+1)}-1)$ for all integer $k$. Assuming $A^q = -A$, this reduces to $\lambda^{(q+1)(2k+1)} = 1$, which implies that $q-1 \mid 2k + 1$, a contradiction since $q$ is odd.
\end{example}

\begin{example}
The analogous example for even $q$ is the \emph{projective triad} of side $(q+2)/2$ (see \cite{JWPH:79}). This set consists of $(3q+2)/2$ points with $(q+2)/2$ points on each of three concurrent lines (point of intersection included) such that the line joining any two points of the set on different lines of the triad meets the third line at a point of the set. A representation of a projective triad which includes the point $(0,0,1)$ has affine points $(x,\mathrm{tr}(x),1)$, where ``$\mathrm{tr}$'' is the absolute trace function from $GF(q^2)$ to $GF(2)$, and infinite points $(1,\frac{1}{a},0) \cup (1,0,0)$ where $\mathrm{tr}(a) = 1$. The critical cone of the proper start flock which arises from this example consists only of points other than $(0,0,1,0)$ on the lines $y = cx$ where $\mathrm{tr}(c) = 0$. Note that $(0,1,0,0)$ is never in the critical cone since the absolute trace function is not a permutation. The width of the base of the critical cone is therefore $\frac{q}{2}$ and the critical cone is thin, but again any larger cone would be wide.

In the special case of $PG(2,2^{2e})$, the flock $\mathcal{F}(t, \mathrm{tr}(t), 0)$ has a critical cone which contains the points of $x^qy = z^{q+1}$ for $q = 2^e$ except for $(0,1,0,0)$ in the carrier plane $w = 0$. This follows easily since an affine point of the line $y = cx$ other than $(0,0,1,0)$ which lies on this curve must satisfy $c = (\frac{z}{x})^{q+1} \in GF(q)$. This implies that in $GF(q^2)$ we have $\mathrm{tr}(c) = 0$. Note that if the relative trace function (from $GF(q^2)$ to $GF(q)$) instead of the absolute trace function is used in the definition of the flock, we would have obtained a $q$-linear (instead of bilinear) flock which is of the Megyesi-Holder type.
\end{example}



\end{document}